\documentclass[11pt]{amsart}
\usepackage[francais,english]{babel}
\usepackage{palatino}
\usepackage{amsfonts}
\usepackage{amssymb}
\usepackage{amscd}
\usepackage[breaklinks,bookmarksopen,bookmarksnumbered]{hyperref}
\usepackage[dvips]{graphicx}

\newtheorem{theorem}{Theorem}[section]
\newtheorem{proposition}[theorem]{Proposition}
\newtheorem{corollary}[theorem]{Corollary}
\newtheorem{lemma}[theorem]{Lemma}
\theoremstyle{definition}

\newtheorem{remark}[theorem]{Remark}

\newtheorem{conjecture/question}[theorem]{Conjecture/Question}

\newtheorem{remark/definition}[theorem]{Remark/Definition}
\newtheorem{terminology/notation}[theorem]{Terminology/Notation}

\def\PP{{\textbf P}}
\def\OO{\mathcal{O}}

\def\cA{\mathcal{A}}
\def\F{\mathcal{F}}

\def\W{\mathcal{W}}

\def\T{\mathfrak{Th}}

\def\cM{\mathcal{M}}
\def\cR{\mathcal{R}}

\def\cZ{\mathcal{Z}}

\def\cX{\mathcal{X}}
\def\bm{{\bf{M}}}

\def\Pic0{{\rm Pic}^0(X)}

\def\ww{\overline{\mathcal{W}}}
\def\ff{\overline{\mathcal{F}}}
\def\mm{\overline{\mathcal{M}}}

\def\cc{\overline{\mathcal{C}}}
\def\tf{\widetilde{\mathcal{F}}}

\def\Thet{{\bf{\Theta}}}
\begin{document}
\title{The universal theta divisor over the moduli space of curves}

\author[G. Farkas]{Gavril Farkas}

\address{Humboldt-Universit\"at zu Berlin, Institut F\"ur Mathematik,  Unter den Linden 6
\hfill \newline\texttt{}
 \indent 10099 Berlin, Germany} \email{{\tt farkas@math.hu-berlin.de}}
\thanks{}

\author[A. Verra]{Alessandro Verra}
\address{Universit\'a Roma Tre, Dipartimento di Matematica, Largo San Leonardo Murialdo \hfill
\indent 1-00146 Roma, Italy}
 \email{{\tt
verra@mat.unirom3.it}}

\maketitle

\begin{abstract} By computing the class of the universal antiramification locus of the Gauss map, we obtain a complete birational classification by Kodaira dimension of the universal theta divisor over the moduli space of curves.
\end{abstract}

\begin{otherlanguage}{french}
\begin{abstract}
En calculant la classe universelle du lieu d'antiramification de l'application de Gauss, nous obtenons
une classification birationnelle compl\`ete par la dimension de Kodaira du diviseur th\^eta universel sur l'espace des modules de
courbes.
\end{abstract}
\end{otherlanguage}

The universal theta divisor over the moduli space $\cA_g$ of principally polarized abelian varieties of dimension $g$ is the divisor ${\bf{\Theta}}_g$ inside the universal abelian variety $\xi:\mathcal{X}_g\rightarrow \cA_g$ characterized by the following two properties:
\vskip 2pt

\noindent (i) ${{\bf{\Theta}}_{g}} _{| \xi^{-1}([A, \Theta])}=
\Theta$, for every principally polarized abelian variety $[A, \Theta]\in \cA_g$.

\noindent (ii) The restriction $s^*({\bf{\Theta}}_g)$ along the zero section $s:\cA_g\rightarrow \mathcal{X}_g$ is trivial on $\cA_g$.
 \vskip 3pt
 The study of the geometry of ${\bf{\Theta}}_g$  mirrors
that of $\cA_g$ itself. It is known that ${\bf{\Theta}}_g$ is unirational for $g\leq 4$; the case $g\leq 3$ is classical,  for $g=4$ we refer to \cite{Ve}. Whenever $\cA_g$ is known to be of general type (that is, in the range $g\geq 7$, cf. \cite{Fr}, \cite{Mu}, \cite{T}), one can use Viehweg's additivity theorem \cite{Vi} for the fibre space ${\bf{\Theta}}_g\rightarrow \cA_g$ to conclude that ${\bf{\Theta}}_g$ is of general type as well. Our first result concerns the birational type of $\Thet_5$.
\begin{theorem}\label{gen5}
The universal theta divisor $\Thet_5$ is uniruled.
\end{theorem}
By making use of the generically finite Prym map $P:\cR_6\rightarrow \cA_5$, we show that the pull-back of the universal theta divisor over the Prym moduli space $\cR_6$, that is,
$$\Thet_5^{\mathrm{p}}:=\Thet_5\times_{\cA_5} \cR_6$$ is uniruled. We sketch the idea of the proof and refer to Section 4 for details.

We fix a general element $[C, \eta]\in \cR_6$, inducing an \'etale double cover $f:\tilde{C}\rightarrow C$. The canonical curve $C\subset \PP^5$ can be viewed as a quadratic section of a smooth quintic del Pezzo surface $S\subset \PP^5$. A general element $L$ on the theta divisor $\Xi(C, \eta)$ of the Prym variety $P(C, \eta)$ is a line bundle $L$ on $\tilde{C}$ such that $\mbox{Nm}_f(L)=K_C$ and $h^0(\tilde{C}, L)=2$.
We associate to this data a rank $3$ quadric $Q_L\in \PP \mbox{Sym}^2 H^0(K_C)$ everywhere tangent to the canonical curve $C\subset \PP^5$ along a Prym canonical divisor, that is, $$C\cdot Q_L=2d_L,$$ where $d_L\in |K_C\otimes \eta|$. The pencil in $|-2K_S|$ generated by the curves $C$ and $S\cdot Q_L$ induces a rational curve in the universal Prym theta divisor $\Thet_5^{\mathrm{p}}$.
\vskip 3pt

Next we obtain a birational classification of the universal theta divisor $$\T_g:=\cM_g\times_{\cA_g} {\bf{\Theta}}_g$$
over the moduli space of curves. If $[C]\in \cM_g$ is a smooth curve, the Abel-Jacobi map $C_{g-1}\rightarrow \mbox{Pic}^{g-1}(C)$ provides a resolution of singularities of the theta divisor $\Theta_C$ of the Jacobian of $C$. Thus one may regard the degree $g-1$ universal symmetric product $\cc_{g, g-1}:=\mm_{g, g-1}/\mathfrak S_{g-1}$ as a birational model of $\T_g$ (having only finite quotient singularities), and ask for the place of $\T_g$ in the classification of varieties. We provide a complete answer to this question. For small genus, $\T_g$ enjoys rationality properties.
\begin{theorem}\label{class}
$\T_g$ is unirational for $g\leq 9$ and uniruled for $g\leq 11$.
\end{theorem}
The first part of the theorem is a consequence of Mukai's work \cite{M1}, \cite{M2} on representing canonical curves with general moduli as linear sections of certain homogeneous varieties. When $g\leq 9$, there exists a Fano variety $V_g\subset \PP^{N_g}$ of dimension $n_g:=N_g-g+2$ and index $n_g-2$, such that general $1$-dimensional complete intersections of $V_g$ are canonical curves $[C]\in \cM_g$ having general moduli. The  correspondence
$$\Sigma:=\Bigl\{\bigl((x_1, \ldots, x_{g-1}), \Lambda\bigr)\in V_g^{g-1}\times G(g, N_g+1): x_i\in \Lambda, \mbox{ for } i=1, \ldots, g-1\Bigr\}$$ maps dominantly onto $\T_g$ via the map $\bigl((x_1, \ldots, x_{g-1}), \Lambda\bigr)\mapsto [V_g\cap \Lambda, x_1+ \cdots+x_{g-1}]$.
Since $\Sigma$ is a Grassmann bundle over the rational variety $V_g^{g-1}$, it follows that $\T_g$ is unirational in the range $g\leq 9$. The cases $g=10, 11$ are settled by the observation that in this range the space $\mm_{g, g-1}$ is uniruled, see \cite{FP}, \cite{FV2}.

For the remaining genera, we achieve a complete classification. This is the main result of the paper.
\begin{theorem}\label{gentyp}
The universal theta divisor $\T_g$ is a variety of general type for $g\geq 12$.
\end{theorem}

We also have a birational classification theorem for the universal degree $n$ symmetric product $\cc_{g, n}:=\mm_{g, n}/\mathfrak{S}_n$  for all $1\leq n\leq g-2$, and  refer to Section 3 for details. Our results are complete in degree  $g-2$ and less precise as $n$ decreases. Similarly to Theorem \ref{gentyp},  the nature of $\cc_{g, g-2}$ changes when $g=12$:
\begin{theorem}\label{minus2}
The universal degree $g-2$ symmetric product $\cc_{g, g-2}$ is  uniruled for $g<12$ and a variety of general type for $g\geq 12$.
\end{theorem}

The proof of Theorem \ref{gentyp} relies on the calculation of the universal \emph{antiramification divisor} class of the Gauss map. For a curve $C$ of genus $g$, let $\gamma:C_{g-1}\dashrightarrow \bigl(\PP^{g-1}\bigr)^{\vee}$ be the Gauss map given by $\gamma(D):=\langle D\rangle$ for each divisor $D\in C_{g-1}$ with $h^0(C, \OO_C(D))=1$. The branch divisor $\mbox{Br}(\gamma)\subset (\PP^{g-1})^{\vee}$  is the dual of the canonical curve $C\subset \PP^{g-1}$. The closure in $C_{g-1}$ of the ramification divisor $\mbox{Ram}(\gamma)$ is the locus of divisors $D\in C_{g-1}$ such that $\mbox{supp}(D)\cap \mbox{supp}(K_C(-D))\neq \emptyset$. \footnote{The description of the ramification divisor of the Gauss map given in \cite{ACGH} p. 247 is erroneous.} The \emph{antiramification} divisor $\mbox{Ant}(\gamma)$, defined by the following equality of divisors
$$\gamma^*(\mbox{Br}(\gamma))=\mbox{Ant}(\gamma)+2\cdot \mbox{Ram}(\gamma),$$ splits into the locus of non-reduced divisors $\Delta_C:=\{2p+D: p\in C,\ D\in C_{g-3}\}$ and the locus of divisors $D\in C_{g-1}$ such that $K_C(-D)$ has non-reduced support.  Globalizing this construction over $\cM_g$, we are led to consider the universal antiramification divisor
$$\mathfrak{Ant}_g:=\Bigl\{[C, x_1, \ldots, x_{g-1}]\in \cM_{g, g-1}: \exists p\in C \ \ \mathrm{ with } \ H^0\bigl(K_C(-x_1-\cdots-x_{g-1}-2p)\bigr)\neq 0\Bigr\}.$$
We have the following formula for the class of the closure $\overline{\mathfrak{Ant}}_g$ of $\mathfrak{Ant}_g$ in $\mm_{g, g-1}$:
\begin{theorem}\label{gaussram}
The class in $\mathrm{Pic}(\mm_{g, g-1})$ of the closure of the antiramification locus is\footnote{We are grateful to Scott Mullane for pointing out a confusing typo in the published version of the paper.}
$$[\overline{\mathfrak{Ant}}_g]= -4(g-7)\lambda-2\delta_{\mathrm{irr}}-\sum_{i=0}^g\sum_{s=0}^{i-1}\Bigl((2g-3)s^2-(4gi+2g-10i+1)s+2gi^2-7i^2+2gi+i+2\Bigr) \delta_{i: s}.$$
\end{theorem}

Note that the coefficient of $-\delta_{0:2}$ in this expression equals $12g-22$. In this formula, we also use the convention $\delta_{0:1}=\delta_{g:g-2}:=\sum_{i=1}^{g-1} \psi_i$, hence by substituting, the coefficient of $\psi_1+\cdots+\psi_{g-1}$ in the expression $[\overline{\mathfrak{Ant}}_g]$ equals $4(g-2)$. By construction, $\overline{\mathfrak{Ant}}_g$ is $\mathfrak S_{g-1}$-invariant. So it descends to an effective divisor $\widetilde{\mathfrak{Ant}}_g$ on $\cc_{g, g-1}$ which, as it turns out, spans an extremal ray of the effective cone $\mbox{Eff}(\cc_{g, g-1})$. The universal theta divisor is equipped with the involution $\tau:\cc_{g, g-1}\dashrightarrow \cc_{g, g-1}$,
$$\tau\bigl([C, x_1+\cdots+x_{g-1}]\bigr):=[C, y_1+\cdots+y_{g-1}],$$ where $\OO_C(y_1+\cdots+y_{g-1}+x_1+\cdots+x_{g-1})=K_C$.  Then $\widetilde{\mathfrak{Ant}}_g$ is the pull-back of the boundary divisor $\widetilde{\Delta}_{0:2}\subset \cc_{g, g-1}$ under this map. Since the extremality of $\widetilde{\Delta}_{0: 2}$ is easy to establish, the following result comes naturally:
\begin{theorem}\label{extremal}
The effective divisor $\widetilde{\mathfrak{Ant}}_g$ is covered by irreducible curves $\Gamma_g\subset \cc_{g, g-1}$ such that $\Gamma_g\cdot \widetilde{\mathfrak{Ant}}_g<0$. In particular $\widetilde{\mathfrak{Ant}}_g\in \mathrm{Eff}(\cc_{g, g-1})$ is a non-movable extremal effective divisor.
\end{theorem}
The curves $\Gamma_g$ have a simple modular construction. One fixes a general linear series $A\in W^2_{g+1}(C)$, so that in particular $A$ is complete and has only ordinary ramification points. The general point of $\Gamma_g$ corresponds to an element $[C, D]\in \cc_{g, g-1}$, where  $D\in C_{g-1}$ is an effective divisor such that $H^0\bigl(C, A\otimes \OO_C(-2p-D)\bigr)\neq 0$, for some point $p\in C$, that is, $D$ is the residual divisor cut out by a tangent line to the degree $g+1$ plane model of $C$ given by $A$. We refer to Section 2 for details.
\vskip 3pt

The proofs of Theorems \ref{gentyp} and \ref{minus2} rely on two ingredients. First, we use our result \cite{FV2}, stating that for $g\geq 4$, the singularities of $\cc_{g, n}$ impose no \emph{adjoint conditions}, that is, pluricanonical forms defined on the smooth locus of
$\cc_{g, n}$ extend to a smooth model of the symmetric product. Precisely, if
$\epsilon:\widetilde{\mathcal{C}}_{g, n}\rightarrow \cc_{g, n}$ denotes any resolution of singularities,  then for any $\ell\geq 0$, there is a group isomorphism
$$\epsilon^*: H^0\bigl((\cc_{g, n})_{\mathrm{reg}}, K_{\cc_{g, n}^{\otimes \ell}}\bigr)\stackrel{\cong}\rightarrow H^0\bigl(\widetilde{\mathcal{C}}_{g, n}, K_{\widetilde{\mathcal{C}}_{g, n}^{\otimes \ell}}\bigr).$$ In particular, $\T_g$ is of general type when the canonical class $K_{\cc_{g, g-1}}\in \mbox{Pic}(\cc_{g, g-1})$ is big. This makes the problem of understanding the effective cone of $\cc_{g, g-1}$ of some importance. If $\pi:\mm_{g, g-1}\rightarrow \cc_{g, g-1}$ is the quotient map, the Hurwitz formula implies
\begin{equation}\label{pullbackcan}
\pi^*(K_{\cc_{g, g-1}})\equiv K_{\mm_{g, g-1}}-\delta_{0: 2}
\in \mathrm{Pic}(\mm_{g, g-1}).
\end{equation}
The sum $\sum_{i=1}^{g-1} \psi_i\in \mathrm{Pic}(\mm_{g, g-1})^{\mathfrak S_{g-1}}$ of cotangent tautological classes descends to a big and nef class on $\cc_{g, g-1}$ (cf. Proposition \ref{bigclass}). So in order to conclude that $\T_g$ is of general type, it suffices to exhibit an effective divisor $\mathfrak D$ on $\mathrm{Eff}(\cc_{g, g-1})$, such that
\begin{equation}\label{requirement}
\pi^*(K_{\cc_{g, g-1}})\in \mathbb Q_{>0}\Bigl\langle \sum_{i=1}^{g-1}\psi_i\Bigr\rangle+\phi^*\mathrm{Eff}(\mm_{g})+ \mathbb Q_{\geq 0}\Bigl\langle \lambda, \ \pi^*([\mathfrak D]), \ \delta_{i: s}: i\geq 0, s\geq 2\Bigr\rangle.
\end{equation}
In this formula, $\phi:\mm_{g, g-1}\rightarrow \mm_g$ denotes the morphism forgetting the marked points, and we refer to Section 1 for the standard notation for boundary divisor classes on $\mm_{g, n}$. Comparing condition (\ref{requirement}) against the formula for $K_{\cc_{g, g-1}}$ given by (\ref{canmgn}), if one writes $\pi^*([\mathfrak D])= a\lambda- b_{\mathrm{irr}}\delta_{\mathrm{irr}}+c\sum_{i=1}^{g-1}\psi_i-\sum_{i, s}b_{i: s}\delta_{i: c}\in \mathrm{Pic}(\mm_{g, g-1})$, the inequality
\begin{equation}\label{ineqnec}
3c<b_{0: 2}
\end{equation}
is a necessary condition for the existence of a divisor $\mathfrak D$ satisfying (\ref{requirement}). It is straightforward to unravel the geometric significance of the condition
(\ref{ineqnec}). If $[C]\in \cM_g$ is a general curve, there is a rational map $u:C_{g-1}\dashrightarrow \cc_{g, g-1}$ given by restriction. Denoting by $x, \theta \in N^1(C_{g-1})_{\mathbb Q}$ the standard generators of the N\'eron-Severi group of the symmetric product, the inequality
(\ref{ineqnec}) characterizes precisely those divisors $\mathfrak{D}\in \mbox{Pic}(\cc_{g, g-1})$ for which $u^*([\mathfrak{D}])$ lies in the fourth quarter of the $(\theta, x)$-plane (see \cite{K1} for details on the effective cone of $C_{g-1}$). The divisor $\mathfrak D\subset \cc_{g, g-1}$ playing this role is precisely $\overline{\mathfrak{Ant}}_g$.

\vskip 4pt

We explain briefly how Theorem \ref{gaussram} implies the statement about the Kodaira dimension of $\cc_{g, g-1}$. We choose an effective divisor class $[D]= a\lambda-\sum_{i=0}^{\lfloor \frac{g}{2}\rfloor} b_i\delta_i \in \mbox{Eff}(\mm_g)$ on the moduli space of curves, with $a, b_i\geq 0$, having slope
$s=s(D):=\frac{a}{\mathrm{min}_i b_i}$ as small as possible. Then note that the following linear combination
$$\pi^*(K_{\cc_{g, g-1}})-\frac{1}{6g-11}\Bigl(\frac{3}{2}[\overline{\mathfrak{Ant}}_g]-(12g-25)\phi^*([D])-\sum_{i=1}^{g-1}\psi_i-\bigl((84g-185)-(12g-25)s\bigr)\lambda\Bigr)
$$
is expressible as a positive combination of boundary divisors on $\mm_{g, g-1}$. Since, as already pointed out, the class $\sum_{i=1}^{g-1} \psi_i\in \mathrm{Pic}(\mm_{g, g-1})$ descends to a big class on $\cc_{g, g-1}$, one obtains the following:
\begin{corollary}\label{condition}
For all $g$ such that the slope of the moduli space of curves satisfies the inequality
$$s(\mm_g):=\mathrm{inf}_{D\in \mathrm{Eff}(\mm_g)} s(D)<\frac{84g-185}{12g-25}, $$
the universal theta divisor $\T_g$ is of general type.
\end{corollary}
The bound appearing in Corollary \ref{condition} holds precisely when $g\geq 12$; for $g$ such that $g+1$ is composite, the inequality $s(\mm_g)\leq 6+12/(g+1)$ is well-known, and $D$ can be chosen to be a Brill-Noether divisor $\mm_{g, d}^r$ corresponding to curves with a $\mathfrak g^r_d$ when the Brill-Noether number $\rho(g, r, d)=-1$, cf. \cite{EH}. When $g+1$ is prime and $g\neq 12$, then in practice $g=2k-2\geq 16$, and $D$ can be chosen to be the Gieseker-Petri $\overline{\mathcal{GP}}_{g, k}^1$ consisting of curves $C$ possessing a pencil $A\in W^1_k(C)$ such that the Petri map
$\mu_0(C, A): H^0(C, A)\otimes H^0(C, K_C\otimes A^{\vee})\rightarrow H^0(C, K_C)$ is not an isomorphism. When $g=12$, one has to use the divisor constructed on $\mm_{12}$ in \cite{FV1}. Finally, when $g\leq 11$ it is known that $s(\mm_g)\geq 6+12/(g+1)$ and inequality (\ref{condition}) is not satisfied.
In fact, as already pointed out $\kappa(\T_g)=-\infty$ in this range.
\vskip 3pt

The proof of Theorem \ref{minus2} proceeds along similar lines, and relies on finding an explicit $\mathfrak S_{g-2}$-invariant extremal ray of the cone of effective divisors on $\mm_{g, g-2}$. A representative of this  ray is characterized by the geometric condition that the marked points appear in the same fibre of a pencil of degree $g-1$. One can construct such divisors on all moduli spaces $\mm_{g, n}$ with $1\leq n\leq g-2$, cf. Section 3.
\begin{theorem}\label{classminus2}
The closure inside $\mm_{g, g-2}$ of the locus
$$\F_{g, 1}:=\Bigl\{[C, x_1, \ldots, x_{g-2}]\in \cM_{g, g-2}: \exists A\in W^1_{g-1}(C) \  \mbox{ with  }\  H^0\bigl(C, A(-\sum_{i=1}^{g-2} x_i)\bigr)\neq 0\Bigr\}$$
is a non-movable, extremal effective divisor on $\mm_{g, g-2}$. Its class is given by the formula:
$$[\ff_{g, 1}]= -(g-12)\lambda+(g-3)\sum_{i=1}^{g-2}\psi_i-\delta_{\mathrm{irr}}-\frac{1}{2}\sum_{s=2}^{g-2}s(g-4+sg-2s)\ \delta_{0: s}-\cdots\in \mathrm{Pic}(\mm_{g, g-2}) .$$
\end{theorem}
Note that again, inequality (\ref{ineqnec}) is satisfied, hence $\ff_{g, 1}$ can be used to prove that $K_{\cc_{g, g-2}}$ is big. Moreover, $\ff_{g, 1}$ descends to an extremal divisor $\widetilde{\mathcal{F}}_{g, 1}\in  \mathrm{Eff}(\cc_{g, g-2})$. In fact, we shall show that $\widetilde{\mathcal{F}}_{g, 1}$ is swept by curves intersecting its class negatively.
\vskip 4pt

\section{Cones of divisors on universal symmetric products}

The aim of this section is to establish certain facts about boundary divisors on $\mm_{g, n}$ and $\cc_{g, n}$, see \cite{AC} for a standard reference. As in \cite{FV2}, we denote the coarse moduli space of a Deligne-Mumford stack $\bf{M}$ by $\cM$. We denote by $C_n$ the $n$-th symmetric product of a smooth curve $C$.

For an integer $0\leq i\leq \lfloor \frac{g}{2}\rfloor$ and a subset
$T\subset \{1, \ldots, n\}$, let $\Delta_{i: T}$ be the closure in $\mm_{g, n}$ of the locus of $n$-pointed curves $[C_1\cup C_2, x_1, \ldots, x_n]$, where $C_1$ and $C_2$ are smooth curves of genera $i$ and $g-i$ respectively meeting transversally in one point, and the marked points lying on $C_1$ are precisely those indexed
by $T$. We define the class $\delta_{i: T}:=[\Delta_{i: T}]_{\mathbb Q}\in \mbox{Pic}(\mm_{g, n})$. For $0\leq i\leq \lfloor \frac{g}{2}\rfloor$ and $0\leq s\leq g$, we set
$$\Delta_{i: s}:=\sum_{|T|=s}\delta_{i: T}, \ \ \ \  \delta_{i: s}:=[\Delta_{i: s}]_{\mathbb Q}\in \mbox{Pic}(\mm_{g, n}).$$
 By convention, $\delta_{0: s}:=\emptyset$, for $s<2$, and $\delta_{i: s}:=\delta_{g-i: n-s}$.
If $\phi:\mm_{g, n}\rightarrow \mm_g$ is the morphism forgetting the marked points, we set $\lambda:=\phi^*(\lambda)$ and $\delta_{\mathrm{irr}}:=\phi^*(\delta_{\mathrm{irr}})$, where $\delta_{\mathrm{irr}}:=[\Delta_{\mathrm{irr}}]\in \mbox{Pic}(\mm_g)$ denotes the class of the locus of irreducible nodal curves. Furthermore, $\psi_1, \ldots, \psi_n\in \mbox{Pic}(\mm_{g, n})$ are the cotangent classes corresponding to the marked points. The canonical class of $\mm_{g, n}$ is computed via Kodaira-Spencer theory:
\begin{equation}\label{canmgn}
K_{\mm_{g, n}}\equiv 13\lambda-2\delta_{\mathrm{irr}}+\sum_{i=1}^n \psi_i-2\sum_{T\subset \{1, \ldots, n\}\atop i\geq 0} \delta_{i: T}-\delta_{1: \emptyset}\in \mathrm{Pic}(\mm_{g, n}).
\end{equation}
 Let $\cc_{g, n}:=\mm_{g, n}/\mathfrak S_n$ be the universal symmetric product and let $\pi:\mm_{g, n}\rightarrow \cc_{g, n}$ and $\varphi: \cc_{g, n}\rightarrow \mm_g$ be respectively the projection and the forgetful map, so that $\phi=\varphi\circ \pi$. We denote by $\widetilde{\lambda}, \widetilde{\delta}_{\mathrm{irr}}, \widetilde{\delta}_{i: c}:=[\widetilde{\Delta}_{i: c}]\in \mbox{Pic}(\cc_{g, n})$ the divisor classes on the symmetric product pulling back  to the same symbols on $\mm_{g, n}$.  Clearly, $\pi^*(\widetilde{\lambda})=\lambda, \ \pi^*(\widetilde{\delta}_{\mathrm{irr}})=\delta_{\mathrm{irr}}$, $\pi^*(\widetilde{\delta}_{i: c})=\delta_{i: c}$; in the case $i=0, c=2$, this reflects the branching of the map $\pi$ along the divisor $\widetilde{\Delta}_{0: 2}\subset \cc_{g, n}$.  Following \cite{FV2}, let $\mathbb L$ denote the line bundle on $\cc_{g, n}$ having fibre $$\mathbb L[C, x_1+ \cdots+ x_n]:=T_{x_1}^{\vee}(C)\otimes \cdots \otimes T_{x_n}^{\vee}(C)$$ over a point $[C, x_1+\cdots+x_n]:=\pi([C, x_1, \ldots, x_n])\in \cc_{g, n}$. We set $\widetilde{\psi}:=c_1(\mathbb L)$, and note:
\begin{equation}\label{tildepsi}
\pi^*(\widetilde{\psi})=\sum_{i=1}^n \Bigl(\psi_i-\sum_{i\in T\subset \{1, \ldots, n\}} \delta_{0: T}\Bigr)=\sum_{i=1}^n \psi_i-\sum_{s=2}^n s\ \delta_{0:s}\in \mathrm{Pic}(\mm_{g, n}).
\end{equation}
\begin{proposition} For $g\geq 3$ and $n\geq 0$, the morphism $\pi^*:\mathrm{Pic}(\cc_{g, n})_{\mathbb Q}\rightarrow \mathrm{Pic}(\mm_{g, n})_{\mathbb Q}$ is injective. Furthermore, there is an isomorphism of groups $\mathrm{Pic}(\cc_{g, n})_{\mathbb Q}\stackrel{\cong}\rightarrow N^1(\cc_{g, n})_{\mathbb Q}$.
\end{proposition}
\begin{proof} The first assertion is an immediate consequence of the existence of the norm morphism  $\mathrm{Nm}_{\pi}:\mathrm{Pic}(\mm_{g, n})\rightarrow \mathrm{Pic}(\cc_{g, n})$, such that $\mathrm{Nm}_{\pi}(\pi^*(L))=L^{\otimes{\mathrm{deg}(\pi)}}$, for every $L\in \mathrm{Pic}(\cc_{g, n})$. The second part comes from the isomorphism $\mathrm{Pic}(\mm_{g, n})_{\mathbb Q}\stackrel{\cong}\rightarrow  N^1(\mm_{g, n})_{\mathbb Q}$, coupled with the commutativity of the diagrams relating the Picard and N\'eron-Severi groups of $\mm_{g, n}$ and $\cc_{g, n}$ respectively.
\end{proof}
One may thus identify $\mathrm{Pic}(\cc_{g, n})_{\mathbb Q}\cong \mathrm{Pic}(\mm_{g, n})_{\mathbb Q}^{\mathfrak S_n}$.  The Riemann-Hurwitz formula applied to the branched covering $\pi:\mm_{g, n}\rightarrow \cc_{g, n}$  yields  $$\pi^*(K_{\cc_g, n})=K_{\mm_{g, n}}-\delta_{0: 2}\equiv 13\lambda+\sum_{i=1}^n \psi_i-2\delta_{\mathrm{irr}}-3\delta_{0: 2}-2\sum_{s=3}^n \delta_{0: s}-\cdots.$$
As expected, the sum of cotangent classes descends to a big line bundle on $\cc_{g, n}$.
\begin{proposition}\label{bigclass}
The divisor class $N_{g, n}:=\widetilde{\psi}+\sum_{s=2}^{n} s\widetilde{\delta}_{0:s}\in \mathrm{Eff}(\cc_{g, n})$ is big and nef.
\end{proposition}
\begin{proof} The class $N_{g, n}$ is characterized by the property that  $\pi^*(N_{g, n})=\sum_{i=1}^n \psi_i$. This is a nef class on $\mm_{g, n}$,  in particular, $N_{g, n}$ is nef on $\cc_{g, n}$. To establish that $N_{g, n}$ is big, we express it as a combination of effective classes and the class $\widetilde{\kappa}_1\in \mathrm{Pic}(\cc_{g, n})$, where
$$\pi^*(\widetilde{\kappa}_1)=\kappa_1=12\lambda+\sum_{i=1}^n \psi_i-\delta_{\mathrm{irr}}-\sum_{i=0}^{[g/2]}\sum_{s\geq 0} \delta_{i:s}\in \mathrm{Pic}(\mm_{g, n}).$$ Since $\pi^*(\widetilde{\kappa}_1)$ is ample on $\mm_{g, n}$, it follows that $\widetilde{\kappa}_1$ is ample as well. To finish the proof, we exhibit a suitable effective class on $\mm_{g, n}$ having negative $\lambda$-coefficient. For that purpose, we choose
$\W_{g, n}\subset \mathcal{C}_{g, n}$ to be the locus of effective divisors having a Weierstrass point in their support. For $i=1, \ldots, n$, we denote by $\sigma_i:\mm_{g, n}\rightarrow \mm_{g, 1}$ the morphism forgetting all but the $i$-th point, and let $$[\ww]= -\lambda+{g+1\choose 2}\psi-\sum_{i=1}^{g-1} {g-i+1\choose 2}\delta_{i: 1}\in \mathrm{Eff}(\mm_{g, 1}),$$ be the class of the divisor of Weierstrass points on the universal curve. Then one finds
$$[\pi^*(\overline{\W}_{g, n})]= \sum_{i=1}^n \sigma_i^*([\ww])=-n\lambda+{g+1\choose 2}\sum_{i=1}^n \psi_i-{g+1\choose 2}\sum_{s=2}^n s\delta_{0: s}-\cdots\in \mathrm{Pic}(\mm_{g, n}),$$
and $[\overline{\W}_{g, n}]= -g\widetilde{\lambda}+{g+1 \choose 2}\widetilde{\psi}-\sum_{i=1}^{\lfloor \frac{g}{2}\rfloor}\sum_{s\geq 0} b_{i:s}\widetilde{\delta}_{i:s}$, where $b_{i:s}>0$.
One checks that $N_{g, n}$ can be written as a $\mathbb Q$-combination with positive coefficients of the ample class $\widetilde{\kappa}_1$, the effective class $[\overline{\W}_{g, n}]$ and other boundary divisor classes. In particular,
$N_{g, n}$ is big.
\end{proof}

\section{The universal antiramification locus of the Gauss map}

We begin the calculation of the divisor $\overline{\mathfrak{Ant}}_g$, and for a start we consider its restriction $\mathfrak{Ant}_g$ to $\cM_{g, g-1}$. Recall that $\mathfrak{Ant}_g$ is defined as the closure of the locus of pointed curves $[C, x_1, \ldots, x_{g-1}]\in \cM_{g, g-1}$, such that there exists a holomorphic form on $C$
vanishing at $x_1, \ldots, x_{g-1}$ and having an unspecified double zero.

Let $u:{\bf{M}}_{g, g-1}^{(1)}\rightarrow {\bf{M}}_{g, g-1}$ be the universal curve over the stack of $(g-1)$-pointed smooth curves and denote by $\bigl([C, x_1, \ldots, x_{g-1}], p\bigr)\in \cM_{g, g-1}^{(1)}$ a general point, where $[C, x_1, \ldots, x_{g-1}]\in \cM_{g, g-1}$ and $p\in C$ is an arbitrary point.
For $i=1, \ldots, g-1$, let $\Delta_{ip}\subset \cM_{g, g-1}^{(1)}$ be the diagonal divisor given by the equation $p=x_i$. For $i=1, \ldots, g-1$ we consider as before the projections $\sigma_i: \bm_{g, g-1}^{(1)}\rightarrow \bm_{g, 1}$ and $\sigma_p:\bm_{g, g-1}^{(1)}\rightarrow \bm_{g, 1}$, induced by forgetting all marked points except $x_i$ and $p$ respectively. Then set $$K_i:=\sigma_i^*(\omega_{\phi})\in \mathrm{Pic}(\bm_{g, g-1}^{(1)})\ \mbox{ and } K_p:=\sigma_p^*(\omega_{\phi})\in \mbox{Pic}(\bm_{g, g-1}^{(1)}).$$
Furthermore, we consider the following cartesian diagram of stacks
\[
\begin{CD}
{\mathcal X}@>{q}>> {\bf{M}}_{g, g-1}^{(1)}\\
@VV{f}V@VV{}V\\
{\bf{M}}_{g, 1}^{}@>{\phi}>>{\bf{M}}_{g}\\
\end{CD}
\]
in which all the morphisms are smooth and $\phi$ (hence also $q$) is
proper. For $1\leq i\leq g-1$ there are tautological sections $r_i: {\bf{M}}_{g, g-1}^{(1)}\rightarrow \mathcal{X}$
as well as $r_p: {\bf{M}}_{g, g-1}^{(1)}\rightarrow \mathcal{X}$, and we set $E_i:=\mbox{Im}(r_i), E_p:=\mbox{Im}(r_p)$. Thus $\{E_i\}_{i=1}^{g-1}$ and
$E_p$ are relative divisors over $q$.

For a point $\bigl([C, x_1, \ldots, x_{g-1}], p\bigr)\in \cM_{g, g-1}^{(1)}$, we denote  $D:=\sum_{i=1}^{g-1}x_i+2p\in C_{g+1}$, and have the following exact sequence:
$$0\rightarrow \frac{H^0\bigl(\OO_C(D)\bigr)}{H^0(\OO_C)}\rightarrow H^0\bigl(\OO_D(D)\bigr)
\stackrel{\alpha_D}\rightarrow H^1(\OO_C)\rightarrow H^1\bigl(\OO_C(D)\bigr)\rightarrow 0.$$
In particular, the morphisms $\alpha_D$ globalize to a morphism of vector bundles over ${\bf{M}}_{g, g-1}^{(1)}$
$$\alpha: \cA:=q_*\Bigl(\OO_{\mathcal{X}}\bigl(\sum_{i=1}^{g-1} E_i+2E_p\bigr)/\OO_{\mathcal{X}}\Bigr)\rightarrow R^1q_*\OO_{\mathcal{X}}.$$
The subvariety $\cZ:=\bigl\{\bigl([C, x_1, \ldots, x_{g-1}], p\bigr)\in \cM_{g, g-1}^{(1)}: H^0\bigl(K_C(-2p-\sum_{i=1}^{g-1} x_i)\bigr)\neq 0\bigr\}$
is  the non-surjectivity locus of  $\alpha$ and $\mathfrak{Ant}_g:=u_*(\cZ)\subset \cM_{g, g-1}$.
The class of $\cZ$ is equal to
$$[\cZ]=c_2\Bigl(\cA^{\vee}-\bigl(R^1q_*\OO_{\mathcal{X}}\bigr)^{\vee}\Bigr)=c_2\Bigl(-q_!\OO_{\mathcal{X}}(\sum_{i=1}^{g-1} E_i+2E_p)\Bigr)\in A^2({\bf{M}}_{g, g-1}^{(1)}),$$
where the last term can be computed by Grothendieck-Riemann-Roch:
$$\mathrm{ch}\Bigl(q_!\OO_{\mathcal{X}}\bigl(\sum_{i=1}^{g-1} E_i+2E_p\bigr)\Bigr)=q_*\Bigl[\Bigl(\sum_{k\geq 0}\frac{(\sum_{i=1}^{g-1} E_i+2E_p)^k}{k!}\Bigr)\cdot
\Bigl(1-\frac{c_1(\omega_q)}{2}+\frac{c_1^2(\omega_q)}{12}+\cdots\Bigr)\Bigr],$$
and we are interested in evaluating the terms of degree $1$ and $2$ in this expression. The result of applying GRR to the morphism $q$ can be summarized as follows:
\begin{lemma}\label{grr}
One has the following relations in $A^*(\bm_{g, g-1}^{(1)})$:\newline
\noindent (i) $$\mathrm{ch}_1\Bigl(q_*\bigl(\OO_{\cX}(\sum_{i=1}^{g-1} E_i+2E_p)\bigr)\Bigr)=\lambda-\sum_{i=1}^{g-1} K_i-3K_p+2\sum_{i=1}^{g-1} \Delta_{ip}.$$

\noindent (ii) $$\mathrm{ch}_2\Bigl(q_*\bigl(\OO_{\cX}(\sum_{i=1}^{g-1} E_i+2E_p)\bigr)\Bigr)=\frac{5}{2}K_p^2+\frac{1}{2}\sum_{i=1}^{g-1}K_i^2-2\sum_{i=1}^{g-1} (K_i+K_p)\cdot \Delta_{ip}.$$
\end{lemma}
\begin{proof} We apply systematically the push-pull formula and the following identities:
$$E_{i}^2=-E_{i}\cdot q^*(K_i), \ E_p^2=-E_p\cdot q^*(K_p), \ E_i\cdot c_1(\omega_q)=E_i\cdot q^*(K_i), \ E_p\cdot c_1(\omega_q)=E_p\cdot q^*(K_p),$$
$$E_i\cdot E_j=0 \ \mbox{ for } i\neq j, \ E_i\cdot E_p=E_i\cdot q^*(\Delta_{ip}), \mbox{ and }\ q_*(c_1^2(\omega_q))=12\lambda.$$
\end{proof}
\begin{proposition}\label{interior}
One has $[\mathfrak{Ant}_g]=-4(g-7)\lambda+(4g-8)\sum_{i=1}^{g-1}\psi_i\in \mathrm{Pic}(\bm_{g, g-1})$.
\end{proposition}
\begin{proof}
We apply the results of Lemma \ref{grr}, as well as the formulas from \cite{HM} p. 55, in order  to estimate the push-forward under $u$ of the degree $2$ monomials in tautological classes. Setting $\F:=q_*\bigl(\OO_{\cX}(\sum_{i=1}^{g-1} E_i+2E_p)\bigr)$, we obtain that
$$u_*\bigl(\mathrm{ch}_1^2(\F)\bigr)=-(8g-116)\lambda+(8g-24)\sum_{i=1}^{g-1}\psi_i, \mbox{ and } \ \ u_*\bigl(\mathrm{ch}_2(\F)\bigr)=30\lambda-4\sum_{i=1}^{g-1}\psi_i,$$
hence $[\mathfrak{Ant}_g]=u_*\bigl(\mathrm{ch}_1^2(\F)-2\mathrm{ch}_2(\F)\bigr)/2$, and the claimed formula follows at once.
\end{proof}

We proceed now towards proving Theorem \ref{gaussram} and  expand the class $[\overline{\mathfrak{Ant}}_g]$ in the standard basis of the Picard group, that is,
$$[\overline{\mathfrak{Ant}}_g]= a\lambda+c\sum_{i=1}^{g-1} \psi_i-b_{\mathrm{irr}}\delta_{\mathrm{irr}}-\sum_{i=0}^g\sum_{s=0}^{i-1} b_{i:s}\delta_{i: s}.$$
We have just computed $a=-4(g-7)$ and $c=4(g-2)$. The remaining coefficients are determined by intersecting $\overline{\mathfrak{Ant}}_g$ with curves lying in the boundary of $\mm_{g, g-1}$ and understanding how $\overline{\mathfrak{Ant}}_g$ degenerates. We begin with the coefficient $b_{0: 2}$:
\begin{proposition}\label{b02}
One has the relation $(4g-6)c-(g-2)b_{0: 2}=(4g-2)(g-2)$. It follows that $b_{0: 2}=12g-22$.
\end{proposition}
\begin{proof} We fix a general pointed curve $[C, x_1, \ldots, x_{g-2}]\in \cM_{g, g-2}$ and consider the family
$$C_{x_{g-1}}:=\bigl\{[C, x_1, \ldots, x_{g-2}, x_{g-1}]: x_{g-1}\in C\bigr\}\subset \mm_{g, g-1}.$$
The curve $C_{x_{g-1}}$ is the fibre over $[C, x_1, \ldots, x_{g-2}]$ of the morphism $\mm_{g, g-1}\rightarrow \mm_{g, g-2}$ forgetting the point labelled by $x_{g-1}$. Note that $C_{x_{g-1}}\cdot \psi_i=1$ for $i=1, \ldots, g-2$ and $C_{x_{g-1}}\cdot \psi_{g-1}=3g-4=2g-2+(g-2)$. Obviously
$C_{x_i}\cdot \delta_{0: 2}=g-2$ and the points in the intersection correspond to the case when $x_{g-1}$ collides with one of the fixed points $x_1, \ldots, x_{g-2}$. The intersection of $C_{x_i}$ with the remaining generators of $\mathrm{Pic}(\mm_{g, g-1})$ is equal to zero. We set $A:=K_C\otimes \OO_C(-x_1-\cdots-x_{g-2})\in W^1_g(C)$. By the generality assumption, $h^0(C, A)=2$, and all ramification points of $A$ are simple. Pointed curves in the intersection $C_{x_{g-1}}\cdot \overline{\mathfrak{Ant}}_g$ correspond to points $x_{g-1}\in C$, such that there exists a (ramification) point $p\in C$ with $H^0\bigl(C, A\otimes \OO_C(-2p-x_{g-1})\bigr)\neq 0$. The pencil $A$ carries $4g-2$ ramification points. For each of them there are $g-2$ possibilities of choosing $x_{g-1}\in C$ in the same fibre as the ramification point, hence the conclusion follows.
\end{proof}

Next we determine the coefficient $b_{\mathrm{irr}}$. First we note that the relation
\begin{equation}\label{ellpencil}
a-12b_{\mathrm{irr}}+b_{1:0}=0
\end{equation}
holds. Indeed, the divisor $\overline{\mathfrak{Ant}}_g$ is disjoint from the curve in $\Delta_{1: 0}\subset \mm_{g, g-1}$ obtained from a  fixed pointed curve $[C, x_1, \ldots, x_{g-1}, q]\in \mm_{g-1, g}$, by attaching at the point $q$ a pencil of plane cubics along a section of the pencil induced by one of the $9$ base points.
\begin{proposition}\label{birr}
One has the relation $b_{\mathrm{irr}}=2$.
\end{proposition}
\begin{proof} We fix a general curve $[C, q, x_1, \ldots, x_{g-1}]\in \mm_{g-1, g}$, and we define the family
$$C_{\mathrm{irr}}:=\bigl\{[C/t\sim q, x_1, \ldots, x_{g-1}]: t\in C\}\subset \Delta_{\mathrm{irr}}\subset \mm_{g, g-1}.$$
Then $C_{\mathrm{irr}}\cdot \psi_i=1$ for $i=1, \ldots, g-1$, $C_{\mathrm{irr}}\cdot \delta_{\mathrm{irr}}=-(\mathrm{deg}(K_C)+2)=-2g+2$, and finally $C_{\mathrm{irr}}\cdot \delta_{1: 0}=1$. All other intersection numbers with generators of $\mathrm{Pic}(\mm_{g, g-1})$ equal zero.

We fix an effective divisor $D\in C_e$ of degree $e\geq g$ (for instance $D=q+\sum_{i=1}^{g-1} x_i$). For each pair of points  $(t, p)\in C\times C$, there is an exact sequence on $C$
$$0\rightarrow H^0\bigl(C, K_C(q+t-2p-\sum_{i=1}^{g-1} x_i)\bigr)\rightarrow H^0\bigl(C, K_C(D+q+t-2p-\sum_{i=1}^{g-1} x_i)\bigr)\stackrel{\beta_{t, p}}\rightarrow$$
$$H^0\bigl(D, K_C(D+q+t-2p-\sum_{i=1}^{g-1}x_i)\bigr)\rightarrow H^1\bigl(C, K_C(q+t-2p-\sum_{i=1}^{g-1} x_i)\bigr)\rightarrow 0.$$
The intersection $C_{\mathrm{irr}}\cdot \overline{\mathfrak{Ant}}_g$ corresponds to the locus of pairs $(t, p)\in C\times C$ such that the map $\beta_{t, p}$ is not injective. On the triple product of $C$, we consider the projections $$f: C\times  C\times C\rightarrow C\times C\ \mbox{ and }\ p_1:C\times C\times C\rightarrow C$$ given by
$f(x, t, p)=(t, p)$ and $p_1(x, t, p)= x$,
and we set $A:=K_C(q-\sum_{i=1}^{g-1} x_i)\in \mathrm{Pic}^{g-2}(C)$. We denote by $\Delta_{12}, \Delta_{13}\subset C\times C\times C$ the corresponding diagonals, and finally, introduce the line bundle on $C\times C\times C$
$$\F:=p_1^*(A)\otimes \OO_{C\times C\times C}(\Delta_{12}-2\Delta_{13}).$$
Applying the Porteous formula, one can write
$$C_{\mathrm{irr}}\cdot \overline{\mathfrak{Ant}}_g=c_2(R^1 f_*\F-R^0 f_*\F)=\frac{\mathrm{ch}_1^2(f_!\F)+2\mathrm{ch}_2(f_!\F)}{2}\in A^2(C\times C).$$
We evaluate $\mathrm{ch}_i(f_!\F)$ using GRR applied to the morphism $f$, that is,
$$\mathrm{ch}(f_!\F)=f_*\Bigl[\Bigl(\sum_{a\geq 0}\frac{\bigl(p_1^*(A)+\Delta_{12}-2\Delta_{13}\bigr)^a}{a!}\Bigr)\cdot \Bigl(1-\frac{1}{2}p_1^*(K_C)\Bigr)\Bigr].$$
Denoting by $F_1, F_2\in H^2(C\times C)$ the classes of the fibres, after calculations one finds that
$$\mathrm{ch}_1(f_*\F)=-(g-2)F_1-4(g-2)F_2-2\Delta_C\in H^2(C\times C, \mathbb Q),$$
$$\mathrm{ch}_2(f^*\F)=-2(g-2)\in H^4(C\times C, \mathbb Q),$$
that is, $c_2(R^1 f_*\F-R^0 f_*\F)=4(g-2)(g-1)$. Coupled with (\ref{ellpencil}), this yields $b_{\mathrm{irr}}=2$.
\end{proof}

We are left with the task of determining the coefficient of $\delta_{i: s}$. This requires solving a number of enumerative geometry problems in the spirit of de Jonqui\`eres' formula. We fix integers $0\leq i\leq g$ and $s\leq i-1$ and general pointed curves
$[C, x_1, \ldots, x_s]\in \mm_{i, s}$ and $[D, q, x_{s+1}, \ldots, x_{g-1}]\in \mm_{g-i, g-s}$. We then construct a family of stable curves of genus $g$, by identifying the fixed point $q\in D$ with a variable point, also denoted by $q$, on the component $C$:
$$C_{i: s}:=\big\{[C\cup_q D, x_1, \ldots, x_s, x_{s+1}, \ldots, x_{g-1}]: q\in C\bigr\}\subset \Delta_{i: s}\subset \mm_{g, g-1}.$$
We summarize the non-zero intersection numbers of $C_{i: s}$ with generators of $\mathrm{Pic}(\mm_{g, g-1})$:
$$C_{i:s}\cdot \psi_1=\cdots =C_{i: s}\cdot \psi_s=1,\ \ C_{i: s}\cdot \delta_{i: s-1}=i, \ C_{i:s}\cdot \delta_{i: s}=2i-2+s.$$

\begin{theorem}\label{bis}
We fix integers $0\leq i\leq g$ and $0\leq s\leq i-1$. Then, the following formula holds:
$$b_{i: s}=(2g-3)s^2-(4gi+2g-10i+1)s+2gi^2-7i^2+2gi+i+2.$$
\end{theorem}
In the proof an essential role is played by the following calculation:
\begin{proposition}\label{enum}
Let $i, s$ be integers such that $0\leq s\leq i-1$, and $[C, x_1, \ldots, x_s]\in \cM_{i, s}$ a general pointed curve. The number of pairs $(q, p)\in C\times C$ such that
$$H^0\bigl(C, K_C\otimes \OO_C(-x_1-\cdots-x_s-(i-s-1)q-2p)\bigr)\neq 0,$$
is equal to
$$a(i, s):=2(i-s-1)(2i^3-5i^2+i+2-2i^2s+3is).$$
\end{proposition}
\begin{remark} By specializing, one recovers well-known formulas in enumerative geometry. For instance, $a(3, 0)=56$ is twice the number
of bitangents of a smooth plane quartic, whereas $a(4, 0)=324$ equals the number of canonical divisors of type $3q+2p+x\in |K_C|$, where
$[C]\in \cM_4$. This matches de Jonqui\`eres' formula, cf. \cite{ACGH} page 359.
\end{remark}
\noindent
\emph{Proof of Theorem \ref{bis}.}
We fix a general point $[C\cup_q D, x_1, \ldots, x_{g-1}]\in C_{i:s}\cdot \overline{\mathfrak{Ant}}_g$, corresponding to a point $q\in C$. We shall show that
 $q$ is not one of the marked points $x_1, \ldots, x_s$ on $C$, then give a  characterization of such points and count their number.

 Using the geometric description of $\overline{\mathfrak{Ant}}_g$, the stable curve $C\cup_q D$ possesses a sublimit linear series $\mathfrak g^0_{2g-2}$ of the canonical limit linear series $\mathfrak g^{g-1}_{2g-2}$ on $C\cup_q D$, that is, a pair of non-zero sections of the degree $2g-2$ line bundles on $C$ and $D$ obtained by twisting the respective canonical bundles by divisors supported on $q$, that is,
$$\omega_D\in H^0\bigl(D, K_D\otimes \OO_D(2iq)\bigr) \  \ \mbox{ and } \ \ \omega_C\in H^0\bigl(C, K_C\otimes \OO_C(2g-2i)q\bigr)$$
satisfying the compatibility condition
$\mbox{ord}_q(\omega_C)+\mbox{ord}_q(\omega_D)\geq 2g-2$ coming  from the definition of a limit linear series \cite{EH}. Furthermore, the section $(\omega_C, \omega_D)$ vanishes doubly at an unspecified point $p\in C\cup D$ as well as along the divisor $x_1+\cdots+x_{g-1}$.

\vskip 2pt
We distinguish two cases depending on the position of the point $p$. If $p\in D$ then
$$\mathrm{div}(\omega_C)\geq x_1+\cdots+x_s,\ \ \ \mathrm{div}(\omega_D)\geq x_{s+1}+ \cdots+x_{g-1}+2p.$$
Since the points $q, x_{s+1}, \ldots, x_{g-1}\in D$ are general, we find that $\mbox{ord}_q(\omega_D)\leq i+s-2$. Moreover,
$K_D\otimes \OO_D((i-s+2)q-x_{s+1}-\cdots-x_{g-1})\in W_{g-i+1}^1(D)$ is a pencil, and $p\in D$ is one of its (simple) ramification points. The Hurwitz formula gives $4(g-i)$ choices for such $p\in D$.

By compatibility, $\mbox{ord}_q(\omega_C)\geq 2g-i-s$. A parameter count implies that equality must hold. The condition $H^0\bigl(C, K_C\otimes \OO_C(-x_1-\cdots -x_s-(i-s)q\bigr)\neq 0$, is equivalent to asking that $q\in C$ be a ramification point of
$K_C\otimes \OO_C(-\sum_{j=1}^s x_j)\in W_{2i-2-s}^{i-s-1}(C)$. Since the points $x_1, \ldots, x_s\in C$ are chosen to be general, all ramification points of this linear series are simple and occur away from the marked points. From Pl\"ucker's formula, the number of ramification points equals $(i-s)(i^2-1-is)$. Multiplying this by the number of choices for $p\in D$, we obtain a total contribution of $4(g-i)(i-s)(i^2-is-1)$ to the intersection $C_{i:s}\cdot \overline{\mathfrak{Ant}}_g$ stemming from the case when $p\in D$. The proof that each of these points of intersection is to be counted with multiplicity $1$ is standard and proceeds along the lines of \cite{EH2} Lemma 3.4.
\vskip 4pt

We assume now that $p\in C$. Keeping the notation from above, it follows that $\mbox{ord}_q(\omega_D)=i+s-1$ and $\mbox{ord}_q(\omega_C)=2g-i-s-1$.  Therefore $$0\neq \omega_C\in H^0\Bigl(C, K_C\otimes \OO_C(-\sum_{j=1}^s x_j-(i-s-1)q-2p)\Bigr).$$
The section $\omega_D$ is uniquely determined up to multiplication by scalars, whereas there are $a(i, s)$ choices on the side of $C$, each counted with multiplicity $1$.

In principle, the double zero of the limit holomorphic form could specialize to the point of attachment $q\in C\cap D$, and we prove that this would contradict our generality hypothesis. One considers the semistable curve $X:=C\cup_{q_1} E\cup_{q_2} D$, obtained from $C\cup D$ by inserting a smooth rational component $E$ at $q$, where  $\{q_1\}:=C\cap E$ and $\{q_2\}:=D\cap E$. There also exist non-zero sections $$\omega_D\in H^0\bigl(D, K_D(2iq_2)\bigr), \ \omega_E\in H^0\bigl(E, \OO_E(2g-2)\bigr),\  \mbox{  } \  \omega_C\in H^0\bigl(C, K_C((2g-2i)q_1)\bigr),$$ satisfying
$\mbox{ord}_{q_1}(\omega_C)+\mbox{ord}_{q_1}(\omega_E)\geq 2g-2$ and $\mbox{ord}_{q_2}(\omega_E)+\mbox{ord}_{q_2}(\omega_D)\geq 2g-2$.
Furthermore, $\omega_E$ vanishes doubly at a point $p\in \{q_1, q_2\}^c$. Since $\omega_C$ and $\omega_D$ also vanish along the divisors $x_1+\cdots+x_s$ and $x_{s+1}+\cdots+x_{g-1}$ respectively, one obtains the inequalities $\mbox{ord}_{q_1}(\omega_C)\leq 2g-i-s$ and $\mbox{ord}_{q_2}(\omega_D)\leq i+s-1$: hence by compatibility, $\mbox{ord}_{q_1}(\omega_E)+\mbox{ord}_{q_2}(\omega_E)\geq 2g-3$. This rules out the possibility of a further double zero and shows that this case does not occur.

\vskip 3pt
To summarize, keeping in mind that the $\psi$-coefficient of $[\overline{\mathfrak{Ant}}_g]$ is equal to $4g-8$, we find the relation
\begin{equation}\label{recursion}
(2i-2+s)b_{i: s}-sb_{i: s-1}+s(4g-8)=4(g-i)(i-s)(i^2-is-1)+a(i, s).
\end{equation}
We have by convention $b_{i: -1}=0$ By induction, we find using recursion (\ref{recursion}) the claimed formula for $b_{i: s}$.
 \hfill $\Box$
\vskip 4pt

As already explained, having calculated the class $[\overline{\mathfrak{Ant}}_g]\in \mathrm{Pic}(\mm_{g, g-1})$ and using known bounds on the slope $s(\mm_g)$, one derives that $\T_g$ is of general type when $g\geq 12$. It remains to discuss the last cases in Theorem \ref{class} and thus complete the birational classification of $\T_g$:
\vskip 3pt

\noindent \emph{End of proof of Theorem \ref{class}.} We noted in the Introduction that for $g\leq 9$ the space $\T_g$ is unirational, being the image of a variety which is birational to a Grassmann bundle over the rational variety $V_g^{g-1}$. When $g\in \{10, 11\}$, the space $\mm_{g, g-1}$ is uniruled \cite{FP}. This implies the uniruledness of $\T_g$ as well.
\hfill $\Box$

\section{The Kodaira dimension of $\cc_{g, n}$}
In this section we provide results concerning the Kodaira dimension of the symmetric product $\cc_{g, n}$, where $n\leq g-2$. There are two cases depending on the parity
of the difference $g-n$. When $g-n$ is even, we introduce a subvariety inside $\mathcal{C}_{g, n}$, consisting of divisors $D\in C_n$ which appear in a fibre of a pencil of degree $(g+n)/2$ on a curve $[C]\in \cM_g$. We set integers $g\geq 1$ and $1\leq m\leq g/2$, then consider the locus
$$\mathcal{F}_{g, m}:=\Bigl\{[C, x_1, \ldots, x_{g-2m}]\in \cM_{g, g-2m}:\exists A\in W^1_{g-m}(C)
\mbox{ with  }\ H^0\bigl(A(-\sum_{j=1}^{g-2m}x_j)\bigr)\neq 0\Bigr\}.$$ A parameter count shows that $\mathcal{F}_{g, m}$ is expected to be an effective divisor on $\mm_{g, g-2m}$. We shall confirm this expectation, then compute the class of its closure in $\mm_{g, g-2m}$.

\begin{theorem}\label{misc}
Fix integers $g\geq 1$ and $1\leq m\leq g/2$, then set $n:=g-2m$ and $d:=g-m$. The class of the
compactification inside $\mm_{g, g-2m}$ of the divisor $\mathcal{F}_{g, m}$ is given by the formula:
$$[\ff_{g, m}]= \Bigl(\frac{10n}{g-2}{g-2
\choose d-1}-\frac{n}{g}{g\choose
d}\Bigr)\lambda+\frac{n-1}{g-1}{g-1\choose d-1}\sum_{j=1}^n
\psi_j-\frac{n}{g-2}{g-2\choose d-1}\delta_{\mathrm{irr}}-$$ $$-\sum_{s=2}^n\frac{s(n^2-g+sgn-sn)}{2(g-1)(g-d)}{g-1 \choose
d}\delta_{0: s}-\cdots \in \mathrm{Pic}(\mm_{g, n}).
$$
\end{theorem}
\begin{proof} We fix a general curve $[C]\in \cM_g$ and consider the incidence correspondence
$$\Sigma:=\{(D, A)\in C_{g-2m}\times W^1_{g-m}(C): H^0(C, A\otimes \OO_C(-D))\neq 0\},$$
together with the projection $\pi_1:\Sigma\rightarrow C_{g-2m}$. It follows from \cite{F1} Theorem 0.5, that $\Sigma$ is pure of dimension $g-2m-1 \bigl(=\rho(g, 1, g-m)+1\bigr)$.
To conclude that $\ff_{g, m}$ is a divisor inside $\mm_{g, g-2m}$, it suffices to show that the general fibre of the map $\pi_1$ is finite, which implies that $\phi^{-1}([C])\cap \ff_{g, m}$ is a divisor in $\phi^{-1}([C])$; we also note that the fibre $\phi^{-1}([C])$ is isomorphic to the $n$-th Fulton-Macpherson configuration space of $C$. We specialize to the case $D=(g-2m)\cdot p$, where $p\in C$. One needs to show that for a general curve $[C]\in \cM_g$, there exist finitely many pencils $A\in W^1_{g-m}(C)$ with $h^0\bigl(C, A\otimes \OO_C(-(g-m)p)\bigr)\geq 1$, for some point $p\in C$. This follows from \cite{HM} Theorem B, or alternatively, by letting $C$ specialize to a flag curve consisting of a rational spine and $g$ elliptic tails, in which case the point $p$ specializes to a $(g-2m)$-torsion point on one of the elliptic tails (in particular it cannot specialize to a point on the spine). For each of these points, the pencils in question are in bijective correspondence to points in a transverse intersection of Schubert cycles in $G(2, g-m+1)$. In particular their number is finite.

In order to compute the class $[\ff_{g, m}]$, we expand it in the usual basis of $\mathrm{Pic}(\mm_{g, n})$
$$[\ff_{g, m}]= a\lambda+c\sum_{i=1}^{g-2m}\psi_i-b_{\mathrm{irr}}\delta_{\mathrm{irr}}-\sum_{i, s\geq 0} b_{i:s}\delta_{i:s},$$
then note that the coefficients $a, c$ and $b_{\mathrm{irr}}$ have been computed in \cite{F2} Theorem 4.9. The coefficient $b_{0:2}$ is determined by intersecting $\ff_{g, m}$ with a fibral curve
$$C_{x_n}:=\{[C, x_1, \ldots, x_{n-1}, x_n]: x_n\in C\}\subset \mm_{g, n},$$
corresponding to a general $(n-1)$-pointed curve $[C, x_1, \ldots, x_{n-1}]\in \mm_{g, n-1}$. By letting the points $x_1, \ldots, x_{n-1}\in C$ coalesce to a point $q\in C$, points in the  intersection $C_{x_n}\cdot \ff_{g, m}$ are in $1:1$ correspondence with points $x_n\in C$ such that $h^0\bigl(C, A(-(n-1)q-x_n)\bigr)\geq 1$. The number of these points equals $(g-2m-1){g\choose m}$ \ (see \cite{HM} Theorem A), that is,
$$(2g+2n-4)c-(n-1)b_{0:2}=C_{x_n}\cdot \ff_{g, m}=
$$
$$(m+1)\ \#\Bigl\{A\in W^1_{g-m}(C): h^0\bigl(C, A\otimes \OO_C(-(g-2m-1)q)\bigr)\geq 1\Bigr\}=(g-2m-1){g\choose m},$$
which determines $b_{0:2}$.
The coefficients $b_{0: s}$ are computed recursively, by exhibiting an explicit test curve $\Gamma_{0: s}\subset \Delta_{0: s}$ which is disjoint from $\ff_{g, m}$. We fix a general element $[C, q, x_{s+1}, \ldots, x_n]\in \mm_{g, n+1-s}$ and a general $s$-pointed rational curve $[\PP^1, x_1, \ldots, x_s]\in \mm_{0, s}$. We glue these curves along a moving point $q$ lying on the rational component:
$$\Gamma_{0: s}:=\{[\PP^1\cup_q C, x_1, \ldots, x_s, x_{s+1}, \ldots, x_n]: q\in \PP^1\}\subset \Delta_{0: s}\subset \mm_{g, n}.$$
Clearly, $\Gamma_{0: s}\cdot \ff_{g, m}=s c+(s-2)\ b_{0: s}-s\ b_{0: s-1}$. We claim $\Gamma_{0: s}\cap \ff_{g, m}=\emptyset$.
Assume that on the contrary, one can find a point $q\in \PP^1$ and a limit linear series $\mathfrak g^1_d$ on $\PP^1\cup_q C$,
$$l=\bigl((A, V_C), (\OO_{\PP^1}(d), V_{\PP^1})\bigr) \in G^1_d(C)\times G^1_d(\PP^1),$$
together with sections $\sigma_C\in V_C$ and $\sigma_{\PP^1}\in V_{\PP^1}$, satisfying $\mbox{ord}_q(\sigma_C)+\mbox{ord}_q(\sigma_{\PP^1})\geq d$ and
$$\mbox{div}(\sigma_C)\geq x_{s+1}+\cdots+x_n, \ \ \  \   \mbox{div}(\sigma_{\PP^1})\geq x_1+\cdots+x_s.$$
Since $\sigma_{\PP^1}\neq 0$, one finds that $\mbox{ord}_q(\sigma_{\PP^1})\leq g-m-s$, hence by compatibility, $\mbox{ord}_q(\sigma_C)\geq s$.
We claim that this is impossible, that is, $H^0\bigl(C, A\otimes \OO_C(-sq-x_1-\cdots-x_n)\bigr)\neq 0$, for every $A\in W^1_{g-m}(C)$. Indeed, by letting all points $x_{s+1}, \ldots, x_n, q\in C$ coalesce, the statement $H^0\bigl(C, A\otimes \OO_C(-(g-2m)\cdot q)\bigr)=0$, for a general $[C, q]\in \mm_{g, 1}$ is a consequence of the "pointed" Brill-Noether theorem as proved in \cite{EH} Theorem 1.1. This shows that
$$0=\Gamma_{0:s}\cdot \ff_{g, m}=sc+(s-2)b_{0: s}-sb_{0:s-1},$$
for $3\leq s\leq n$, which determines recursively all coefficients $b_{0:s}$. The remaining coefficients $b_{i:s}$ with $1\leq i\leq [g/2]$ can be determined via similar test curve calculations, but we skip these details.
\end{proof}

Keeping the notation from the proof of Theorem \ref{misc}, a direct consequence is the calculation of the class of the divisor $\F_{g, m}[C]:=\pi_1(\Sigma)$ inside  $C_{g-2m}$. This offers an alternative proof of \cite{Mus} Proposition III; furthermore the proof of Theorem \ref{misc} answers in the affirmative the question raised
in \cite{Mus}, concerning whether the cycle $\F_{g, m}[C]$ has expected dimension, and thus, it is a divisor on $C_{g-2m}$.

We denote by $\theta \in H^2(C_{g-2m}, \mathbb Q)$ the class of the pull-back of the theta divisor, and by $x\in H^2(C_{g-2m}, \mathbb Q)$ the class of the locus
$\{p_0+D:D\in C_{g-2m-1}\}$ of effective divisors containing a fixed point $p_0\in C$. For a very general curve $[C]\in \cM_g$, the group
$N^1(C_{g-2m})_{\mathbb Q}$ is generated by $x$ and $\theta$, see \cite{ACGH}.

Let $\widetilde{\mathcal{F}}_{g, m}$ be the effective divisor on $\cc_{g, g-2m}$ to which $\ff_{g, m}$ descends, that is, $\pi^*(\widetilde{\mathcal{F}}_{g, m})=\ff_{g, m}$. The class of $\tf_{g, m}$ is completely determined by Theorem \ref{misc}.

\begin{corollary}\label{symmclass}
 Let $[C]\in \cM_g$ be a general curve. The cohomology class of the divisor
$$\F_{g, m}[C]:=\{D\in C_{g-2m}: \exists A\in W^1_{g-m}(C) \ \  such \  that   \ \ H^0(C, A\otimes \OO_C(-D))\neq 0\}$$ is equal to
$(1-\frac{2m}{g}){g\choose m}\bigl(\theta-\frac{g}{g-2m}x\bigr)$. In particular, the class $\theta-\frac{g}{g-2m}x\in N^1(C_{g-2m})_{\mathbb Q}$ is effective.
\end{corollary}
\begin{proof} Let $u:C_{g-2m}\dashrightarrow \cc_{g, g-2m}$ be the rational map given by $$u(x_1+\cdots+x_{g-2m})=[C, x_1+\cdots+x_{g-2m}].$$ Note that $u$ is well-defined outside the codimension $2$ locus of effective divisors with support of length at most $g-2m-2$. We have that $u^*(\widetilde{\delta}_{0: 2})=\delta_C$, where $\delta_C:=[\Delta_C]/2$ is the reduced diagonal. Its class is given by the MacDonald formula, cf. \cite{K1} Lemma 7:
$$\delta_C\equiv -\theta+(2g-2m-1)x.$$
 Furthermore, $u^*(\widetilde{\psi})\equiv \theta+\delta_C+(2m-1)x$, see \cite{K} Proposition 2.7. Thus $\F_{g, m}[C]\equiv
u^*([\tf_{g, m}])$, and the conclusion follows after some calculations.
\end{proof}
\vskip 3pt

The divisor $\tf_{g, m}$ is defined in terms of a correspondence between pencils and effective divisors on curves, and it is fibred in curves as follows: we fix a complete pencil
$A\in W^1_{g-m}(C)$ with only simple ramification points. The variety of secant divisors
$$V_{g-2m}^1(A):=\{D\in C_{g-2m}: H^0(C, A\otimes \OO_C(-D))\neq 0\}$$
is a curve disjoint from the indeterminacy locus of the map $u:C_{g-2m}\dashrightarrow \cc_{g, g-2m}$, see \cite{F1}. We set $\Gamma_{g-2m}(A):=u(V_{g-2m}^1(A))\subset \cc_{g, g-2m}$. By varying $[C]\in \cM_g$ and $A\in W^1_{g-m}(C)$, the curves $\Gamma_{g-2m}(A)$ fill-up the divisor $\tf_{g, m}$. It is natural to test the extremality of $\tf_{g, m}$ by computing the intersection number $\Gamma_{g-2m}(A)\cdot \tf_{g, m}$.
To state the next result in a unified form, we adopt the convention  ${a\choose b}:=0$ whenever $b<0$.

\begin{proposition}\label{extremality}
For all integers $1\leq m< g/2$, we have the formula:
$$\Gamma_{g-2m}(A)\cdot \tf_{g, m}=(m-1){g-m-2\choose m}{g\choose  m}.$$
In particular, $\Gamma_{g-2}(A)\cdot \tf_{g, 1}=0$, and the divisor $\tf_{g, 1}\in \mathrm{Eff}(\cc_{g, g-2})$ is extremal.
\end{proposition}
\begin{proof}
This is an immediate application of Corollary \ref{symmclass}. The class $[V_{g-2m}^1(A)]$ can be computed using Porteous' formula, see \cite{ACGH} p.342:
$$[V^1_{g-2m}(A)]\equiv \sum_{j=0}^{g-2m-1}{-m-1\choose j} \frac{x^j\cdot \theta^{g-2m-j-1}}{(g-2m-1-j)!}\in H^{2(g-2m-1)}(C_{g-2m}, \mathbb Q).$$
Using the push-pull formula, we write $\Gamma_{g-2m}(A)\cdot \tf_{g, m}=\F_{g, m}[C]\cdot [V_{g-2m}^1(A)]$, then estimate the product using the identity $x^k\theta^{g-2m-k}=g!/(2m+k)!\in H^{2(g-2m)}(C_{g-2m}, \mathbb Q)$ for $0\leq k\leq g-2m$. For $m=1$, observe that $\Gamma_{g-2}(A)\cdot \tf_{g, 1}=0$. Since the curves of type $\Gamma_{g-2}(A)$ cover $\tf_{g, 1}$, this implies that $[\tf_{g, 1}]\in \mathrm{Eff}(\cc_{g, g-2})$ generates an extremal ray.
\end{proof}
\vskip 4pt

We can use Theorem \ref{misc} to describe the birational type of $\cc_{g, n}$ when $12\leq g\leq 21$ and $1\leq n\leq g-2$. We recall that when $g\leq 9$, the space $\cc_{g, n}$ is uniruled for all values of $n$. The transition cases $g=10, 11$, as well as the case of the universal Jacobian $\cc_{g, g}$, are discussed in detail in \cite{FV2}. Furthermore $\cc_{g, n}$ is uniruled when $n\geq g+1$; in this case the symmetric product $C_n$ of any curve $[C]\in \cM_g$ is birational to a $\PP^{n-g}$-bundle over the Jacobian $\mbox{Pic}^n(C)$. Our main result is that, in the range described above, $\cc_{g, n}$ is of general type in all the cases when $\mm_{g, n}$ is known to be of general type, see \cite{Log}, \cite{F2}. We note however that the divisors $\ff_{g, m}$ only carry one a certain distance towards a full solution. The classification of $\cc_{g, n}$ is complete only when $n\in \{g-1, g-2, g\}$.

\begin{theorem}\label{cgn}
For integers $g=12, \ldots, 21$, the universal symmetric product $\cc_{g, n}$ is of
general type for all $f(g)\leq n\leq g-1$, where $f(g)$ is described in the
following table.
\begin{center}
\begin{tabular}{c|cccccccccc}
$g$ & 12& 13& 14& 15& 16& 17& 18& 19& 20&
21\\
\hline $f(g)$ & 10& 11& 10& 10& 9&
9&
9& 7& 6& 4\\
\end{tabular}
\end{center}
\end{theorem}
\begin{proof}
The strategy described in the Introduction to prove that  $K_{\cc_{g, g-1}}$ is big applies to the other spaces $\cc_{g, n}$ with $1\leq n\leq g-2$ as well. To show that $\cc_{g, n}$ is of general type, it suffices to produce an effective class on $\cc_{g, n}$ that pulls back via $\pi$ to $a\lambda+c\sum_{i=1}^n \psi_i-b_{\mathrm{irr}}\delta_{\mathrm{irr}}-\sum_{i, s}b_{i: s}\delta_{i: s}\in \mathrm{Eff}(\mm_{g, n})^{\mathfrak{S}_n}$, such that the following conditions are fulfilled:
\begin{equation}\label{sufficient}
\frac{a+s(\mm_g)\bigl(2c-b_\mathrm{irr}\bigr)}{13c}<1 \mbox{ } \   \ \mbox{ and } \ \mbox{  } \ \frac{b_{0:2}}{3c}>1.
\end{equation}

 When $g-n$ is even, we write $g-n=2m$, and for all entries in the table above one can express $K_{\cc_{g, n}}$ as a positive combination of $\sum_{i=1}^n \psi_i$, $[\ff_{g, m}]$, $\varphi^*(D)$, where $D\in \mathrm{Eff}(\mm_g)$, and other boundary classes.

If $g-n=2m+1$ with $m\in \mathbb Z_{\geq 0}$,  for each integer $1\leq j\leq n+1$, we denote by $\phi_j:\mm_{g, n+1}\rightarrow \mm_{g, n}$ the projection forgetting the $j$-th marked point and consider the effective $\mathfrak S_n$-invariant effective $\mathbb Q$-divisor class on $\mm_{g, n}$
$$E:=\frac{1}{n+1}\sum_{j=1}^{n+1} (\phi_j)_*\bigl([\ff_{g, m}]\cdot \delta_{0: \{j, n+1\}}\bigr)\in \mbox{Eff}(\mm_{g, n}).$$
Using Theorem \ref{misc} as well as elementary properties of push-forwards of tautological classes, $K_{\cc_{g, n}}$ is expressible as a positive $\mathbb Q$-combination of boundaries, $E$, a pull-back of an effective divisor on $\mm_g$, and the big and nef class $\sum_{i=1}^n \psi_i$ precisely in the cases appearing in the table.
\end{proof}
\begin{remark} When $g\notin \{12, 16, 18\}$, the bound $s(\mm_g)\leq 6+12/(g+1)$, emerging from the slope of the Brill-Noether divisors, has been used to verify (\ref{sufficient}). In the remaining cases, we employ the better bounds $s(\mm_{12})=4415/642<6+12/13$ (see \cite{FV1}), and $s(\mm_{16})=407/61<6+12/17$ see \cite{F2}, coming from Koszul divisors on $\mm_{12}$ and $\mm_{16}$ respectively. On $\mm_{18}$, we use the estimate $s(\mm_{18})\leq 302/45$ given by the class of the Petri divisor $\overline{\mathcal{GP}}_{18, 10}^1$, see \cite{EH}. Improvements on the estimate on $s(\mm_g)$ in the other cases will naturally translate into improvements in the statement of Theorem \ref{cgn}.
\end{remark}

\section{The universal Prym theta divisor in genus $6$}

The aim of this last section is to establish the uniruledness of the universal theta divisor $\sigma: \Thet_5^{\mathrm{p}}\rightarrow \cR_6$ over the moduli space $\cR_6$ classifying pairs $[C, \eta]$, where $C$ is a smooth curve of genus $6$ and $\eta\in \mbox{Pic}^0(C)$ is a non-trivial $2$-torsion point. It is proved in \cite{DS} that the Prym map $P:\cR_6\rightarrow \cA_5$ is generically finite of degree $27$, thus in order to conclude that $\Thet_5$ is uniruled it suffices to establish the same conclusion for $\Thet_6^{\mathrm{p}}$.

For a point $[C, \eta]\in \cR_g$, we denote by $f:\tilde{C}\rightarrow C$ the unramified double cover induced by $\eta$, and by $\iota:\tilde{C}\rightarrow \tilde{C}$ the involution interchanging the sheets of $f$. Setting $\PP:=\PP H^0(K_{\tilde{C}})^{\vee}$, we view
$\PP^+:=\PP H^0(K_C)^{\vee}$ as a subset of $\PP$.
If
$$\mbox{Nm}_f:\mbox{Pic}^{2g-2}(\tilde{C})\rightarrow \mbox{Pic}^{2g-2}(C)$$ is the norm map, then $P(C, \eta):=\mbox{Nm}_f^{-1}(K_C)^{+}$ and one has the following realization for the Prym theta divisor:
$$\Xi(C, \eta):=\bigl\{L\in \mbox{Pic}^{2g-2}(\tilde{C}): \mathrm{Nm}_f(L)=K_C,\ h^0(\tilde{C}, L)\geq 2, \ \ h^0(\tilde{C}, L)\mbox{ is even}\bigr\}.$$
The universal Prym theta divisor $\Thet_g^{\mathrm{p}}$ is the parameter space of triples $[C, \eta, L]$, where $[C, \eta]\in \cR_g$ and $L\in \Xi(C, \eta)$ and $\sigma([C, \eta, L])=[C, \eta]$.

We fix a general element $[C, \eta, L]\in \Thet_g^{\mathrm{p}}$, where $h^0(\tilde{C}, L)=2$. The set of divisors
$$\bigl\{f_*(D): D\in |L|\bigr\}\subset |K_C|$$ can be viewed as a conic in a $(g-1)$-dimensional projective space. Following \cite{Ve2}, to this conic one can associate its dual hypersurface
$Q_L\in \PP\mbox{Sym}^2 H^0(K_C)$ which is a rank $3$ quadric. Alternatively, viewing $L\in W^1_{2g-2}(\tilde{C})$ as a singular point of the Riemann theta divisor $\Theta_{\tilde{C}}$, we consider the projectivized tangent cone $\tilde{Q}_L\supset \tilde{C}$ of $\Theta_{\tilde{C}}$ at the point $\tilde{L}$, and then $Q_L:=\tilde{Q}_L\cap \PP^+$.
In coordinates, if $(s_1, s_2)$ is a basis of $H^0(\tilde{C}, L)$, we have the following concrete description of the two quadrics:
$$\tilde{Q}_L:\bigl( s_1\cdot \iota^*(s_1)\bigr)\cdot \bigl(s_2\cdot \iota^*(s_2)\bigr)-\bigl(s_1\cdot \iota^*(s_2)\bigr)\cdot \bigl(s_2\cdot \iota^*(s_1)\bigr)=0$$
and
$$Q_L:\ 4\bigl( s_1\cdot \iota^*(s_1)\bigr)\cdot \bigl(s_2\cdot \iota^*(s_2)\bigr)-\bigl(s_1\cdot \iota^*(s_2)+s_2\cdot \iota^*(s_1)\bigr)^2=0.$$
We summarize the properties of $Q_L$ and refer to \cite{Ve2} for details:
\begin{proposition}\label{rk3}
For a general point $[C, \eta, L]\in \Thet_g^{\mathrm{p}}$, there exists a rank $3$ quadric $Q_L\subset \PP H^0(K_C)^{\vee}$ satisfying the following properties:
\begin{enumerate}
\item The rulings of $\tilde{Q}_L$ cut out the pencils $L$ and $\iota^*(L)$ on the curve $\tilde{C}$.
\item $Q_L\cdot C=2d_L$, for some Prym canonical divisor $d_L\in |K_C\otimes \eta|$.
\item The divisor $f^*(d_L)$ consists of points $\tilde{x}\in \tilde{C}$ such that $h^0\bigl(\tilde{C}, L(-\tilde{x}-\iota(\tilde{x}))\bigr)=1$.
\item For a general divisor $d\in |K_C\otimes \eta|$, there exists $L\in \Xi(C, \eta)$ such that $d=d_L$.
\end{enumerate}
\end{proposition}

Specializing to the case $g=6$, we are in a position to prove that $\Thet_6^{\mathrm{p}}$ is uniruled.
\vskip 4pt

\noindent \emph{Proof of Theorem \ref{gen5}.}
We fix a general point $[C, \eta, L]\in \Thet_6^{\mathrm{p}}$. There exists a smooth del Pezzo surface $S\subset \PP^5$ containing the canonical image of $C$ and such that $C\in |-2K_S|$. Using Proposition \ref{rk3}, to $L\in \Xi(C, \eta)$, one associates a Prym canonical divisor $d_L\in |K_C\otimes \eta|$ and a quadric $Q_L\subset \PP^5$. The curve $C_L:=S\cap Q_L$ is a smooth curve of genus $6$; since it lies on a rank $3$ quadric it is endowed with a vanishing theta-null induced from the ruling of $Q_L$.  Let $\PP_L\subset |\OO_S(2)|$ be the pencil spanned by $C$ and $C_L$ or, equivalently, the pencil on $S$ that on the curve $C$ cuts out the divisor
$2d_L$.

For each smooth curve $C'\in \PP_L$, the fixed divisor $d_L\in \mbox{Div}(C')$ can be viewed as belonging to the linear system $|K_{C'}\otimes \eta'|$, for a certain point of order two $\eta'\in \mbox{Pic}^0(C')[2]$. This induces a point $[C', \eta']\in \cR_6$. Furthermore, if $f':\tilde{C}'\rightarrow C'$ is the induced covering, since $\tilde{C}'\subset \tilde{Q}_L$, the rulings of $\tilde{Q}_L$ cut out a line bundle $L'\in \mbox{Pic}(\tilde{C}')$ such that $\mbox{Nm}_{f'}(L')=K_{C'}$ and $h^0(\tilde{C}', L')\geq 2$. Furthermore, $h^0(\tilde{C}', L')$ is even, because the parity of the dimension of the space of sections of a line bundle with canonical norm remains constant in families, hence $L'\in \Xi(C', \eta')$. It follows that the family of elements $[C', \eta', L']$ defines a rational curve in $\Thet_6^{\mathrm{p}}$ passing through a general point of the moduli space.
\hfill
$\Box$

\end{document}